\documentclass[12pt,a4paper]{amsart}
\usepackage{amsmath,amssymb,amscd,amsthm}
\usepackage{graphics}
\usepackage{enumerate}
\input xy
\xyoption{all}

\usepackage{comment}
\usepackage{tikz}
\usepackage{setspace}

\usetikzlibrary{arrows}

\usepackage[active]{srcltx}

\usepackage{enumerate}
 \usepackage[colorlinks,final,backref=page,hyperindex]{hyperref}

\usepackage{amsmath,amssymb,xspace,amsthm}
\newtheorem{theorem}{Theorem}[section]
\newtheorem{example}[theorem]{Example}
\newtheorem{remark}[theorem]{Remark}
\newtheorem{lemma}[theorem]{Lemma}
\newtheorem{proposition}[theorem]{Proposition}
\newtheorem{corollary}[theorem]{Corollary}
\numberwithin{equation}{section}

\newcommand{\C}{\mathbb C}

\newcommand{\Z}{\mathbb{Z}}

\def\mg{\mathfrak{g}}

\def\mh{\mathfrak{h}}

\def\sl{\mathfrak{sl}}
\def\gl{\mathfrak{gl}}

\def\br{\mathbf{r}}

\def\b1{\mathbf{1}}
\def\l{\lambda}

\title[Minimal nilpotent finite $W$-algebra]{\bf
Irreducible cuspidal $\mathfrak{sl}_{n+1}$-modules
from finite-dimensional  modules over
 the minimal nilpotent finite $W$-algebra}
\author{Genqiang Liu and Mingjie Li}

\date{\today}

\begin{document}
\begin{abstract}A weight $\mathfrak{sl}_{n+1}$-module with finite-dimensional weight spaces is called cuspidal if every root vector acts injectively on it. We establish  a direct connection between such modules and the minimal nilpotent finite $W$-algebra $W(e)$ of $\mathfrak{sl}_{n+1}$. In particular, every weight space of an irreducible cuspidal module carries a natural structure of an irreducible finite-dimensional $W(e)$-module. Using an explicit embedding $\tau\colon W(e)\hookrightarrow U(\mathfrak{gl}_n)$, we prove that every finite-dimensional irreducible $W(e)$-module is an irreducible quotient of $V(\lambda)^\tau$ for some finite-dimensional irreducible $\mathfrak{gl}_n$-module $V(\lambda)$. We then determine the composition factors of $V(\lambda)^\tau$, characterize its irreducibility in terms of a dot-action orbit, and obtain explicit realizations of all irreducible cuspidal $\mathfrak{sl}_{n+1}$-modules, without using the twisted localization method and the coherent family introduced in \cite{M}.
\end{abstract}
\vspace{5mm}
\maketitle
\vspace{2mm}

\noindent{{\bf Math. Subj. Class.} 2020: 17B05, 17B10, 17B20, 17B35}

\section{Introduction}

Let $\mg$ be a complex semisimple Lie algebra and $e$ a nilpotent element of $\mg$.   The finite $W$-algebra $U(\mg,e)$ was introduced  by Premet \cite{Pr1}  as a quantization of the Slodowy slice through the adjoint orbit of $e$.
In fact, finite $W$-algebras appeared earlier in mathematical physics via BRST cohomology
under a slightly different guise, see for example \cite{BT,DV,Tj}.
 In the case when  $e$ is a principal nilpotent element, it was proven by Kostant \cite{K} that the corresponding finite $W$-algebra  is isomorphic to the center $Z(\mg)$ of $U(\mg)$. In \cite{Pr2}, Premet  determined generators and defining  relations of $U(\mg, e)$ when $e$ is a minimal nilpotent element. In \cite{BK}, Brundan and Kleshchev described any finite $W$-algebra associated to nilpotent orbits in general linear Lie algebras as a factor algebra of a shifted Yangian. This Yangian realization  by generators and relations makes it possible to study the representation theory of finite $W$-algebras in terms of purely combinatorial data. In \cite{BK2}, Brundan and Kleshchev used the relationship between shifted Yangians and the finite $W$-algebras  to study  highest weight representations over the finite $W$-algebras of type $A$.  There is a close connection between finite-dimensional irreducible representations of
$U(\mg, e)$ and primitive ideals of $U(\mg)$. In \cite{L},  Losev proved
that there is a bijection between the primitive ideals of $U(\mg)$ whose associated variety is the
closure of the adjoint orbit of $e$, and the orbits of the component group of the centralizer of $e$ on
the isomorphism classes of finite-dimensional irreducible $U(\mg, e)$-modules.
 This classification result  was refined in  \cite{LO}.
The recent works in \cite{PT,T} focused on the classification of one-dimensional $W$-modules. The category of finite-dimensional modules over the minimal
nilpotent finite $W$-algebras was studied in \cite{Pe}. There are several good surveys of the theory of finite $W$-algebras, see for example \cite{L2,W}.

Fernando and Futorny reduced the classification of irreducible weight modules
with finite-dimensional weight spaces to that of cuspidal modules for Levi
subalgebras \cite{F,Fu}. Among simple complex Lie algebras, cuspidal modules
occur only in types $A$ and $C$. Mathieu classified them using coherent
families and twisted localization \cite{M}. Their categories and block
structures were subsequently investigated in \cite{BKLM,GS1,GS2,GS3,MS2},
and cuspidal modules also provide relaxed highest weight modules for vertex
algebras \cite{ACD,FKR,KR,KR2}. These methods give a powerful classification,
but the structure carried by an individual weight space is less explicit.
The purpose of this paper is to make that structure concrete in type $A$ by
using a minimal nilpotent finite $W$-algebra. Finite $W$-algebras can be used to study Whittaker modules of corresponding Lie algebras, see the appendix given by Serge Skryabin in \cite{Pr1}. But they are rarely used to study weight modules, especially cuspidal modules. In this paper, we offer a direct connection between cuspidal $\sl_{n+1}$-modules and finite-dimensional modules over the minimal nilpotent finite $W$-algebra.

Let  $e=e_{10}+\dots+e_{n0}$  which is a minimal nilpotent element of $\sl_{n+1}$ and $\mh_n$ be the Cartan subalgebra of $\sl_{n+1}$.  In \cite{LL}, it was shown that the finite $W$-algebra $W(e)$ of $\sl_{n+1}$ associated with $e$ is isomorphic to the centralizer $W$ of the subalgebra $\mh_n\oplus (\sum_{i=1}^n\C e_{0i})$ in some localization $U_S$
of $U(\sl_{n+1})$, see Theorem \ref{iso-thw}. Moreover, Theorem 19 in \cite{LL} established an equivalence between a block with a generalized central character  of the cuspidal $\mathfrak{sl}_{n+1}$-module category
and a block of the finite-dimensional $W(e)$-module category.
By Theorem \ref{main-iso-th},   every weight space of an irreducible cuspidal module
is an irreducible finite-dimensional $W$-module, see Corollary \ref{weight-space}. So irreducible cuspidal $\mathfrak{sl}_{n+1}$-modules were reduced to finite-dimensional irreducible modules over  $W(e)$. Our main algebraic tool in this paper is an explicit embedding $
 \tau\colon W(e)\hookrightarrow U(\gl_n).
$.

The paper is organized as follows. In Section \ref{Sec.2}, we review in detail several facts about
 $W(e)$, including the centralizer realization of $W(e)$  from \cite{LL} and  the  injective homomorphism $\tau$  from $W$  to $U(\gl_n)$.
In Section \ref{Sec.3},  we study finite-dimensional irreducible modules over the  algebra $\sigma\tau(W)$, where $\sigma=\textup{exp}(-\textup{ad}X)$ is the inner automorphism of $U(\gl_n)$ and
$X=\sum_{k=1}^{n-1} e_{kn}$.  Using  highest-weight module theory, we show that if $M$ is a finite-dimensional irreducible  $\sigma\tau(W)$-module, then there is a finite-dimensional irreducible $\gl_n$-module $V(\l)$ such that $M$ is isomorphic to an irreducible quotient
of  $V(\l)$, see Theorem \ref{finite-1}. It should be mentioned  that \cite{BK2} provides a sufficient and necessary condition for  irreducible highest weight modules over finite $W$-algebras of type $A$ to be finite-dimensional, extending the  theory of Yangians due to Tarasov \cite{Ta} and Drinfeld \cite{Dr}. Our method is more specific and elementary.
 We also show that  restricted to $\sigma\tau(W)$,  the composition length of $V(\l)$ is at most $2$, and determine  all its irreducible subquotients, see Theorem \ref{irre-factor}.  We also show that these two statements
also hold for $W$, see Theorems \ref{finite-w-mod} and \ref{irre-factor2}.
In Section \ref{Sec.4}, using the results in Section 3, we give specific realizations  of
all irreducible cuspidal $\sl_{n+1}$-modules, see Theorem \ref{Realization-theorem} and Corollary \ref{Classification-theorem2}.

In this paper, we denote by $\Z$ and $\C$ the sets of integers and complex numbers, respectively. All vector spaces and algebras are over $\C$. For  $k\in \Z$,  define $\Z_{\geq k}=\{m\in \Z\mid m\geq k\}$.
For a Lie algebra $\mathfrak{g}$ we denote by $U(\mathfrak{g})$ its universal enveloping algebra.
 We write $\otimes$ for
$\otimes_{\mathbb{C}}$. For an $\alpha=(\alpha_1,\dots,\alpha_n)\in \C^n$, let $|\alpha|=\alpha_1+\dots+\alpha_n$. Let $\{e_1,\dots, e_n\}$ be the standard basis of $\mathbb{C}^n$. For an associative algebra $A$ (resp. a Lie algebra $\mg$), let $Z(A)$ (resp. $Z(\mg)$) be the center of $A$ (resp. $U(\mg)$). An algebra homomorphism  $\chi: Z(A)\rightarrow \C$ is called a central character. An  $A$-module $M$
 is said to have  generalized central character $\chi$ if $z-\chi(z)$ acts locally
 nilpotently on it for any $z\in Z(A)$.

\section{Preliminaries}\label{Sec.2}

In this section, we collect some necessary preliminaries, including
the algebra $W(e)$, a centralizer realization of $W(e)$,
cuspidal modules over $\sl_{n+1}$, and a {\it Miura type map} $\tau: W\rightarrow U(\gl_n)$.

\subsection{The definition of  $W(e)$}

For a fixed  $n\in \Z_{\geq 2}$, let $\gl_{n+1}$ be the general linear Lie algebra over $\C$,  i.e., $\gl_{n+1}$ consists of $(n+1)\times (n+1)$-complex matrices.
 Let $e_{ij}$   denote the $(n+1)\times (n+1)$-matrix  unit whose $(i,j)$-entry is $1$ and $0$ elsewhere, $0\leq i,j \leq n$. Then $\{e_{ij}\mid 0\leq i,j \leq n\}$ is a basis of $\gl_{n+1}$. Set $I_{n-1}=\sum_{k=1}^{n-1}e_{kk}$, $I_{n}=\sum_{k=1}^{n}e_{kk}$ and $I_{n+1}=\sum_{k=0}^{n}e_{kk}$.
Let $\mh_n=\oplus_{i=0}^{n-1}\C (e_{ii}-e_{i+1,i+1})$ which is a Cartan subalgebra of $\sl_{n+1}$.  In  $\mh_n$, let $h_i=e_{ii}-\frac{1}{n+1} I_{n+1}$  for any $i\in \{0, 1,\dots,n\}$.
 Define a $\Z$-gradation of $\sl_{n+1}$:
  $$\sl_{n+1}=\mg(-2)\oplus\mg(0)\oplus \mg(2),$$ where $\mg(2)=\oplus_{i=1}^n \C e_{i0}$, $\mg(-2)=\oplus_{i=1}^n \C e_{0i}$, and $\mg(0)\cong \gl_n$ is the subalgebra spanned by $e_{ij}, 1\leq i \neq j\leq n$ and $h_k, k=1,\dots,n$.

 Let $G=SL_{n+1}$ acting on $\sl_{n+1}$ by conjugation.
 Denote by  $\mathcal{N}$ the null cone which consists of all nilpotent elements in $\sl_{n+1}$. Consider the conjugate action of $G$ on $\mathcal{N}$. The set of nilpotent orbits in $\sl_{n+1}$ under this action is naturally a poset with partial order $\leq$ as follows: $O'\leq O$ if and only if $O'\subset \overline{O}$,
 where $\overline{O}$ is the closure of $O$.
 There exists a unique
minimal orbit $O_{\text{min}}$ in $\mathcal{N}$ of the smallest positive dimension.
 Every nilpotent matrix in $\sl_{n+1}$
is conjugate to a matrix in
Jordan canonical form parameterized
by a partition $p= (p_1,p_2,\dots)$ of $n+1$. In particular,
the associated Jordan block of $O_{\text{min}}$ is $(2,1^{n-1})$.

Throughout this paper, we fix a nilpotent element  $e=e_{10}+\dots+e_{n0}$ in
$\sl_{n+1}$.  Since $e$ is conjugate to $e_{10}$, $e$ is a  minimal nilpotent element.

\begin{lemma}The decomposition
$$\sl_{n+1}=\mg(-2)\oplus\mg(0)\oplus \mg(2)$$ is a good $\Z$-grading
for $e$,  see \cite{EK}. That is $e\in \mg(2)$,
$\text{ad}e: \mg(0)\rightarrow\mg(2)$ is surjective, and $\text{ad}e: \mg(-2)\rightarrow\mg(0)$ is injective.
\end{lemma}

\begin{proof} For any $1\leq i\neq k\leq n$,  $[e_{ik}, e]=e_{i0}$. So $\text{ad}e: \mg(0)\rightarrow\mg(2)$ is surjective.
If $[\sum_{k=1}^n a_{k}e_{0k}, e]=\sum_{k=1}^n a_{k}e_{00}-\sum_{k,i=1}^na_{k}e_{ik}=0$, then $a_1=\dots=a_n=0$. Hence $\text{ad}e: \mg(-2)\rightarrow\mg(0)$ is injective.
\end{proof}

The map $\theta:\mg(-2) \rightarrow \C, X\mapsto \text{tr}(Xe)$ defines a one dimensional
$\mg(-2) $-module $\C_{\b1}:=\C v_{\b1}$. It is easy to see that $\theta(e_{0i})=1$,  for any $1\leq i \leq n$. Let
$\mathcal{I}_{\b1}$ be the left ideal of $U(\sl_{n+1}) $ generated by
$e_{01}-1,\dots,e_{0n}-1$.
Define  the induced $\sl_{n+1}$-module
(called generalized Gelfand-Graev module)

\begin{equation}\label{qmod}Q_{\b1} := U(\sl_{n+1}) \otimes_{U(\mg(-2) )}\C_{\b1}\cong U(\sl_{n+1})/\mathcal{I}_{\b1}.\end{equation}

The finite $W$-algebra $ W(e)$ is defined to be
the endomorphism algebra
$$W(e):= \text{End}_{\sl_{n+1}}(Q_{\b1} )^{\text{op}}. $$
This  definition of finite $W$ algebras is due to Premet, see \cite{Pr1}.
 Two different good $\Z$-gradings of $e$ are shown to lead to isomorphic finite
$W$-algebras, see \cite{BG}.
Let $\mathfrak{p}_n:=\mg(2)\oplus\mg(0)$ which is a maximal parabolic subalgebra of $\sl_{n+1}$. The PBW theorem implies that
$$U(\sl_{n+1}) =U(\mathfrak{p}_n)\oplus \mathcal{I}_{\b1} .$$
So as a vector space, $Q_{\b1}\cong U(\mathfrak{p}_n)$.
Applying Frobenius reciprocity, $W(e)$ is isomorphic to the
subalgebra
$$U(\mathfrak{p}_n)^{\mg(-2)}=\{y\in U(\mathfrak{p}_n)\mid
[e_{0i},y]\in \mathcal{I}_{\b1}, \text{for}\  1\leq i\leq n \},$$
since $\xi(v_\b1)$ corresponds to an element in $U(\mathfrak{p}_n)^{\mg(-2)}$,
for any $\xi\in \text{End}_{\sl_{n+1}}(Q_{\b1} )$.


\subsection{Centralizer realization of $W(e)$}
In this subsection, we  recall an explicit realization of the finite $W$-algebra $W(e)$ given in \cite{LL}.

By Lemma 4.2 in \cite{M}, the set
   $$S:=\{e_{01}^{i_1}\dots e_{0n}^{i_n}\mid i_1,\dots, i_n\in\Z_{\geq0}\}$$ is an Ore subset of $U(\sl_{n+1})$.
Let $U_{S}$ be the localization $U(\sl_{n+1})$ with respect to $S$. Let $B$ be the subalgebra of $U_{S}$ generated by $h_{i}, e_{0k}, e_{0k}^{-1}, 1\leq k,i\leq n$.
 Let $W$ be the centralizer of $B$ of $U_S$.
 Define the following distinguished elements in $U_S$:
 \begin{equation}\label{x-w-def}
 \aligned x_{ij}&=e_{ij}e_{0i}e_{0j}^{-1}-h_i,\\
 \omega_k&=e_{k0}e_{0k}+\sum_{j=1}^n
 (e _{kj}-\frac{\delta_{jk}}{n+1}I_{n+1})(h_j-1)e_{0k}e_{0j}^{-1}, \endaligned\end{equation}for all $i,j, k=1,\dots,n$ with
 $i\neq j$.

 The next theorem tells us that $W$ is a tensor product factor of
$U_{S}$, see Lemma 7 and Theorem 8 in \cite{LL}.

 \begin{theorem}\label{main-iso-th}
 We have the following statements.
  \begin{enumerate}[$($a$)$]
 \item $U_{S}\cong B \otimes W $, $Z(U_{S})\cong Z(W)$.
 \item All the ordered monomials in $x_{ij}, \omega_k $, $i, j,k=1,\dots, n $ with
  $i\neq j$ form a basis of $W$
over $\C$.
   \end{enumerate}
\end{theorem}

It seems that the isomorphism in  Theorem  \ref{main-iso-th} may
have some potential connection with the conjecture  of Gelfand and Kirillov on skew-field of $U(\sl_{n+1})$, see \cite{GK}.The following theorem is Theorem 9 in \cite{LL}. Here we give a new proof.

\begin{theorem}\label{iso-thw} We have the algebra isomorphism $W\cong W(e)$.
\end{theorem}
\begin{proof}Let $B_{\b1}=B \otimes_{U(\mg(-2) )}\C_{\b1}$. As a vector space, $B_{\b1}=\C[h_1,\dots,h_n]v_{\b1}$. From
$$(e_{0i}-1)g(h_1,\dots, h_n)v_{\b1}=(g(h_1,\dots, h_i+1,\dots, h_n)-g(h_1,\dots, h_n))v_{\b1},$$ we see that $(e_{0i}-1)$ decreases the degree of $h_i$ in $g(h_1,\dots, h_n)v_{\b1}$. Consequently  any nonzero $B$-submodule of $B_{\b1}$ must contain the generator $v_{\b1}$. So $B_{\b1}$ is an irreducible $B$-module.
Then by Theorem \ref{main-iso-th}, the module $Q_{\b1}$ defined by (\ref{qmod}) is isomorphic to $B_{\b1}\otimes W$.
Since $B$ is an associative $\C$-algebra of countable dimension,
by Schur's Lemma (see Lemma 1 in \cite{PS}), $\text{End}_{ B}(B_{\b1} )\cong \C$.
Then
$W(e)\cong \text{End}_{ B \otimes W}(B_{\b1}\otimes W )^{\text{op}}\cong \text{End}_{W}(W )^{\text{op}}\cong W$.
\end{proof}

\subsection{Weight  $\sl_{n+1}$-modules and
finite-dimensional $W$-modules}

An $\sl_{n+1}$-module $M$ is a weight module if
\[
 M=\bigoplus_{\alpha\in\C^n}M_\alpha,
 \qquad
 M_\alpha=\{v\in M\mid h_iv=\alpha_iv\text{ for }1\leq i\leq n\}.
\]
 Denote $\mathrm{Supp}(M):=\{\alpha\in \C^n \mid M_\alpha\neq0\}$.
A weight
$\sl_{n+1}$-module $M$ is cuspidal if  all weight spaces $M_{\alpha}$ are finite-dimensional, and $e_{ij}: M\rightarrow
M$ is  injective for all $i, j: 0\leq i \neq j \leq n$.  Denote by $\mathcal{C}$ the category of all cuspidal
$\sl_{n+1}$-modules.
The classification of all irreducible weight
modules with finite-dimensional weight spaces over reductive Lie algebras can be reduced to the classification of irreducible cuspidal modules, see \cite{F}. The classification of irreducible cuspidal modules was given in \cite{M} using  the twisted localization method and the coherent family. However, the characterization of weight spaces of any cuspidal module is not clear.
 In this paper, we will use
finite-dimensional irreducible $W$-modules to realize  all irreducible cuspidal
$\sl_{n+1}$-modules.

\begin{corollary}\label{weight-space} If $M=\oplus_{\alpha\in \Z^n} M_{\mu+\alpha}$ is an irreducible cuspidal $\sl_{n+1}$-module,
then $M_{\mu}$ is an irreducible finite-dimensional $W$-module.
\end{corollary}

\begin{proof}Since every root vector acts bijectively on
$M$, $M$ can be extended to be a $U_S$-module. From $[W,\mh_n]=0$,
$WM_{\mu}\subset M_{\mu}$, i.e., $M_{\mu}$ is a $W$-module.
By the isomorphism $U_{S}\cong B \otimes W $, we have that
$M=\C[e_{01}^{\pm1},\dots,e_{0n}^{\pm 1}]\otimes M_{\mu}$.
If $V$ is a nonzero $W$-submodule of  $M_{\mu}$, then $\C[e_{01}^{\pm1},\dots,e_{0n}^{\pm 1}]\otimes V$ is a nonzero
$\sl_{n+1}$-submodule of $M$.
The irreducibility of
$M$ implies that  $V= M_{\mu}$. Therefore $M_{\mu}$ is an irreducible $W$-module.
\end{proof}

\subsection{An algebra homomorphism from $W$ to $U(\gl_n)$}

By Theorem \ref{iso-thw}, we identify $W(e)$ with the subalgebra $W$ of $U_S$. Using the isomorphisms  $W\cong W(e)$ and $ W(e)\cong U(\mathfrak{p}_n)^{\mg(-2)}$, by identifying $h_i$ with $e_{ii}$ for $1\leq i\leq n$,
we have the  algebra homomorphism
 $$\phi:W\rightarrow U(\mathfrak{p}_n)$$ such that:
\begin{equation}\label{x-w-act1}\aligned x_{ij} & \mapsto  e_{ij}-e_{ii},\ 1\leq i\neq j\leq n,\\
 \omega_i
&\mapsto e_{i0}+\sum_{j=1}^n (e_{ij}e_{jj}-e_{ij}),\ 1\leq i\leq n.
\endaligned \end{equation}

Consider the composition of $\phi$ and the canonical projection
from $\mathfrak{p}_n$ to $\mathfrak{p}_n/\mg(2)\cong \gl_n$,
we have the following homomorphism.

\begin{lemma}\label{tau}
There is an injective  homomorphism
 $$\tau:W\rightarrow U(\gl_n)$$ such that:
\begin{equation}\label{x-w-act}\aligned x_{ij} & \mapsto  e_{ij}-e_{ii},\ 1\leq i\neq j\leq n,\\
 \omega_i
&\mapsto \gamma_i:=\sum_{j=1}^n (e_{ij}e_{jj}-e_{ij}),\ 1\leq i\leq n.
\endaligned \end{equation}
\end{lemma}

\begin{proof}We show that $\tau$ is injective. Denote $X_{ij}=e_{ij}-e_{ii}, i\neq j, X_i=e_{ii}$. Then $\{X_{ij}, X_i\mid 1\leq i\neq j\leq n\}$ is a basis of $\gl_n$.
From $$\gamma_i=\sum\limits_{j=1\atop j\neq i}^n ((X_{ij}+X_i)X_j-X_{ij}-X_i)+X_i^2-X_i,$$
it follows that
$$\tau(\omega_1^{r_1}\dots\omega_n^{r_n})=\gamma_1^{r_1}\dots\gamma_n^{r_n}=X_1^{2r_1}\dots X_n^{2r_n}
+g(X_{ij},X_i),$$
where $g(X_{ij},X_i)\in U(\gl_n)$ such that any PBW monomial in $g(X_{ij},X_i)$
is not equal to $X_1^{2r_1}\dots X_n^{2r_n}$.
By (b) in Theorem \ref{main-iso-th}, a general monomial in $W$  has the form
$$\prod_{i\neq j}x_{ij}^{\alpha_{ij}}\prod_{k=1}^n\omega_k^{r_k}.$$
After applying $\tau$, its leading term is
$$\prod_{i\neq j}X_{ij}^{\alpha_{ij}}X_1^{2r_1}\dots X_n^{2r_n},$$
up to lower-order terms.
Hence
 $\tau$ maps nonzero elements
in $W$
to   nonzero elements in $U(\gl_n)$.
Therefore $\tau$ is injective.
\end{proof}

By Lemma \ref{tau}, any  $\gl_n$-module $V$
can be defined to be a $W$-module denoted by $V^\tau$
such that $x\cdot v=\tau(x)v$ for all $x\in W,v\in V$.
We will show that any finite-dimensional irreducible $W$-module is
isomorphic to an irreducible $W$-quotient of the $W$-module $V^\tau$ for some finite-dimensional irreducible $\gl_n$-module $V$.

\begin{remark}The homomorphism $\tau$ is similar to the {\it Miura map}
$\mu$  introduced by Premet, see Remark 2.2 in \cite{Pr2}.
\end{remark}

\section{Finite-dimensional irreducible $W$-modules}\label{Sec.3}

In this section, we study finite-dimensional irreducible $W$-modules using a subalgebra $\sigma\tau(W)\subset U(\gl_n)$ which is isomorphic to $W$. The subalgebra $\sigma\tau(W)$ has a triangular decomposition. Thus we can study finite-dimensional irreducible $\sigma\tau(W)$-modules using the theory of highest weight modules.

\subsection{The structure of  $\sigma\tau(W)$  }

Let $\sigma:=\textup{exp}(-\textup{ad}X)$ which is an  automorphism of $U(\gl_n)$, where $X=\sum_{k=1}^{n-1}e_{kn}$.
It turns out that the algebra $\sigma\tau(W)$ is easier
 than $W$.

The following lemma is used to find good generators of $\sigma\tau(W)$.

 \begin{lemma}\label{sigma} For $1\leq i,j \leq n-1 $, we have
  \begin{enumerate}[$($1$)$]
 \item $\sigma(e_{ij}-e_{in})=e_{ij}$.
 \item $\sigma(e_{nj}-e_{nn})= e_{nj}-\sum_{k=1}^{n-1}e_{kj}$.
  \item $\sigma(\gamma_i)=\gamma_i+e_{in}(I_{n-1}-n+1)
+\sum_{j=1}^{n-1} e_{ij}e_{jn}$.
 \item $\sigma(\gamma_n)=-\sum_{k=1}^{n-1}\sigma( \gamma_k)
+\gamma_n+e_{nn}(I_{n-1}-n+1)+ \sum_{k=1}^{n-1} e_{nk}e_{kn}$.
   \end{enumerate}
\end{lemma}

\begin{proof}
(1)
It is easy to see that $$\aligned \sigma(e_{ij}-e_{in})&=e_{ij}-e_{in}+[e_{ij}-e_{in},X]\\
&= e_{ij}-e_{in}+e_{in}=
e_{ij}.
\endaligned $$

(2)  From $$ [e_{nj}-e_{nn}, X]=e_{nn}-\sum_{k=1}^{n-1}e_{kj}+X$$
and  $$  (\mathrm{ad}X)^2(e_{nj}-e_{nn})=-2X, $$
we have
$$\aligned \sigma(e_{nj}-e_{nn})&=e_{nj}-e_{nn}+e_{nn}-\sum_{k=1}^{n-1}e_{kj}
&= e_{nj}-\sum_{k=1}^{n-1}e_{kj}.
\endaligned $$

(3)
Since $$\aligned\  [\gamma_i,X]&=[ \sum_{j=1}^n (e_{ij}e_{jj}-e_{ij}), \sum_{k=1}^{n-1}e_{kn}]
=\sum_{j=1}^{n-1} (e_{in}e_{jj}+e_{ij}e_{jn})
-e_{in}X-(n-1)e_{in}
\endaligned $$
and
$$\aligned\ (\mathrm{ad}X)^2(\gamma_i) & =  [ \sum_{j=1}^{n-1} (e_{in}e_{jj}+e_{ij}e_{jn}), \sum_{k=1}^{n-1}e_{kn}]
=2 e_{in}X,
\endaligned $$
we obtain
$$\aligned\  \sigma( \gamma_i)
&=\gamma_i+e_{in}(I_{n-1}-n+1)
+\sum_{j=1}^{n-1} e_{ij}e_{jn}.
\endaligned $$

(4)
 From $$\aligned\  [\gamma_n, X]&=[ \sum_{j=1}^n (e_{nj}e_{jj}-e_{nj}), \sum_{k=1}^{n-1}e_{kn}]\\
&= e_{nn}I_{n-1}-\sum_{k=1}^{n-1}\gamma_k-(n-1)e_{nn}+ \sum_{k=1}^{n-1} e_{nk}e_{kn}-e_{nn}X,
\endaligned $$

$$\aligned\  & (\mathrm{ad}X)^2(\gamma_n)\\
 &= [ e_{nn}I_{n-1}-\sum_{k=1}^{n-1}\gamma_k-(n-1)e_{nn}+ \sum_{k=1}^{n-1} e_{nk}e_{kn}-e_{nn}X, X]\\
& = -2XI_{n-1}+2e_{nn}X- 2\sum_{k=1}^{n-1}\sum_{j=1}^{n-1} e_{kj}e_{jn}+2(n-1)X
+2X^2,
\endaligned $$
and
$$\aligned\  -(\mathrm{ad}X)^3(\gamma_n)
&=[ -2XI_{n-1}+2e_{nn}X- 2\sum_{k=1}^{n-1}\sum_{j=1}^{n-1} e_{kj}e_{jn},X]= -6X^2,
\endaligned $$
it follows that
$$\aligned \sigma(\gamma_n)=&
\gamma_n+e_{nn}I_{n-1}-\sum_{k=1}^{n-1}\gamma_k-(n-1)e_{nn}+ \sum_{k=1}^{n-1} e_{nk}e_{kn}\\
& -XI_{n-1}- \sum_{k=1}^{n-1}\sum_{j=1}^{n-1} e_{kj}e_{jn}+(n-1)X,
\endaligned $$
and
$$\aligned \sigma(\gamma_n)+\sum_{k=1}^{n-1}\sigma(\gamma_k)& =
\gamma_n+e_{nn}I_{n-1}-(n-1)e_{nn}+ \sum_{k=1}^{n-1} e_{nk}e_{kn}\\
&\ \ -XI_{n-1}+(n-1)X
+X(I_{n-1}-n+1)\\
& =\gamma_n+e_{nn}(I_{n-1}-n+1)+ \sum_{k=1}^{n-1} e_{nk}e_{kn}.
\endaligned $$
Then (4) follows.

\end{proof}

The following lemma gives generators
of the algebra $\sigma\tau(W)$.

\begin{lemma}\label{w-algebra-3.2}
The elements
$$e_{ij}, e_{ni}, 1 \leq i,j \leq n-1,$$
$$ y_k:=e_{kn}(I_{n}-n)
+\sum_{j=1}^{n-1} e_{kj}e_{jn}  , 1 \leq k \leq n,$$
belong to  $\sigma\tau(W)$. Moreover
$\sigma\tau(W)$ is  generated by these elements.
\end{lemma}

\begin{proof}

Let $\gl_{n-1}$ be the Lie subalgebra spanned by $e_{ij},  1 \leq i,j \leq n-1$. By Lemma \ref{sigma}, we have $\gl_{n-1}, \{e_{ni}| 1 \leq i\leq n-1\}\subset \sigma\tau(W)$.
Then the first stament follows from
$$\aligned\  \sigma( \gamma_k)
& = \sum_{j=1}^n (e_{kj}e_{jj}-e_{kj})
 +e_{kn}(I_{n-1}-n+1)
+\sum_{j=1}^{n-1} e_{kj}e_{jn}\\
  & \equiv   e_{kn}e_{nn}-e_{kn}
 +e_{kn}(I_{n-1}-n+1)
+\sum_{j=1}^{n-1} e_{kj}e_{jn} \ \   (\text{mod}\ U(\gl_{n-1}) )\\
  & \equiv
 y_k
\ \   (\text{mod}\ U(\gl_{n-1}) ),   1 \leq k \leq n-1,
\endaligned $$
and
$$\aligned \sigma(\gamma_n)+\sum_{k=1}^{n-1}\sigma(\gamma_k)
& = \sum_{j=1}^n (e_{nj}e_{jj}-e_{nj})+e_{nn}(I_{n-1}-n+1)+ \sum_{k=1}^{n-1} e_{nk}e_{kn}\\
& =  \sum_{j=1}^{n-1} (e_{nj}e_{jj}-e_{nj})+ y_n.
\endaligned $$
The second statement  follows from (b) in Theorem \ref{main-iso-th}.
\end{proof}
%

%
%
%
%

The next lemma shows that
the structure of  $\sigma\tau(W)$ is very similar to that of $U(\gl_n)$.
The  subspace $\mathrm{Span}\{y_1,\dots,y_{n-1},y_n\}$
is like the $n$-th column of $\gl_n$.

\begin{lemma}\label{lemma-3.3} For $1\leq i,j, k\leq n-1$, we have
\begin{enumerate}
\item  $[e_{ij},y_k]=\delta_{jk}y_i$, $[y_i,y_k]=0$, $[e_{ij},y_n]=0$.
\item $[e_{nj},y_k]=\delta_{jk}y_n-e_{kj}(I_{n-1}-n)
-\sum_{l=1}^{n-1} e_{kl}e_{lj}$.
\item  $[e_{nj},y_n]
=-e_{nj}(I_{n-1}-n)-\sum_{l=1}^{n-1} e_{nl}e_{lj}$.
\end{enumerate}
\end{lemma}

\begin{proof}

 We check (2) and (3). Explicitly
$$\aligned \ [e_{nj},y_k]
=&[e_{nj},e_{kn}(I_{n}-n)
+\sum_{l=1}^{n-1} e_{kl}e_{ln}]\\
=& \delta_{jk}e_{nn}(I_n-n)-e_{kj}(I_n-n)
+\sum_{l=1}^{n-1} \delta_{jk}e_{nl}e_{ln}
+ e_{kj}e_{nn}
-\sum_{l=1}^{n-1} e_{kl}e_{lj}\\
=& \delta_{jk}y_n-e_{kj}(I_n-n)
+ e_{kj}e_{nn}
-\sum_{l=1}^{n-1} e_{kl}e_{lj}\\
=& \delta_{jk}y_n-e_{kj}(I_{n-1}-n)
-\sum_{l=1}^{n-1} e_{kl}e_{lj},
\endaligned$$
and
$$\aligned \  [y_n,e_{nk}]
=&[e_{nn}(I_{n}-n)
+\sum_{l=1}^{n-1} e_{nl}e_{ln},e_{nk}]\\
=& e_{nk}(I_{n}-n)
-e_{nk}e_{nn}
+\sum_{l=1}^{n-1} e_{nl}e_{lk}\\
=& e_{nk}(I_{n-1}-n)+\sum_{l=1}^{n-1} e_{nl}e_{lk}.
\endaligned$$

(1) also follows from direct calculations.
\end{proof}


%
%
%
%
%
%
%
%
%
%
%
%

\begin{example}If $n=2$, $\sigma\tau(W)$ is generated by
$e_{11},y_1,y_2,e_{21}$ subject to the following relation:
\begin{enumerate}
\item  $[e_{11},y_1]=y_1$, $[y_1,e_{21}]=2e_{11}(e_{11}-1)-y_2
$, $[e_{11},e_{21}]=-e_{21}$.
\item $[e_{11},y_2]=0$, $[y_2,e_{21}]
=2e_{21}(e_{11}-1)$
and $C=y_2+e_{11}^2-e_{11}$ is a central element.
\end{enumerate}
In this case, $W$ is isomorphic to the algebra $W_3^{(2)}$ introduced in
\cite{Tj}.
\end{example}

By Theorem \ref{main-iso-th},  Lemmas \ref{w-algebra-3.2} and \ref{lemma-3.3}, the algebra $\sigma\tau(W)$ has the following  triangular decomposition.

\begin{proposition}
$$\sigma\tau(W)=\sigma\tau(W)^-\otimes \sigma\tau(W)^0\otimes \sigma\tau(W)^+,$$where $\sigma\tau(W)^+$ is generated by $y_1,\dots,y_{n-1}, e_{ij}, 1\leq i<j \leq n-1$, $\sigma\tau(W)^0$ is generated by $y_n, e_{ii}, 1\leq i\leq n-1$,
and $\sigma\tau(W)^-$ is generated by $e_{n1},\dots,e_{n,n-1}, e_{ij}, 1\leq j<i \leq n-1$.
\end{proposition}

Similar to $\gl_n$, for $\l_1,\dots,\l_{n-1},\eta_n\in \C$, we  define
the Verma module
\[
 M_{\sigma\tau(W)}(\l_1,\dots,\l_{n-1},\eta_n)
 :=\sigma\tau(W)/J(\l,\eta_n),
\]
where $J(\l,\eta_n)$ is the left ideal generated by
$\sigma\tau(W)^+$, $y_n-\eta_n$, and $e_{ii}-\l_i$
for $1\leq i\leq n-1$. The triangular decomposition implies that this
module has a unique simple quotient, denoted by
$L_{\sigma\tau(W)}(\l_1,\dots,\l_{n-1},\eta_n)$. More generally, a
quotient of $M_{\sigma\tau(W)}(\l_1,\dots,\l_{n-1},\eta_n)$ is called a
highest weight module of highest weight $(\l_1,\dots,\l_{n-1},\eta_n)$.
\

The following lemma will be used in the proof of Proposition \ref{finite-condition}.

\begin{lemma}\label{key-lemma}Let $M$ be a highest weight module over $\sigma\tau(W)$ of the  highest weight $(\l_1,\dots,\l_{n-1},\eta_n)$ and $v^+$ one of its  highest weight vectors. Then for any $k\in \Z_{\geq 1}$, and $s\in\{1,\dots,n-1\}$, we have
$$\aligned \  y_{s}e_{ns}^{k}v^+
&=ke_{ns}^{k-1}\Big( (e_{ss}+1-k)^2+(e_{ss}+1-k)(I_{n-1}-n)\Big)v^+\\
&\ \ +ke_{ns}^{k-1}\Big(\sum_{l=1\atop l\neq s}^{n-1} e_{sl}e_{ls}-y_n\Big)v^+
 -k(k-1)\sum_{l=1\atop l\neq s}^{n-1}e_{ns}^{k-2}e_{nl}e_{ls}v^+.
\endaligned $$
\end{lemma}

\begin{proof}
 We can compute that
 $$\aligned \
y_n e_{ns}^{k-i-1}v^+
& =\sum_{j=0}^{k-i-2}e_{ns}^j
\Big(e_{ns}(I_{n-1}-n)+\sum_{l=1}^{n-1} e_{nl}e_{ls}\Big)e_{ns}^{k-i-2-j}v^++e_{ns}^{k-i-1}y_nv^+\\
&= e_{ns}^{k-i-1}\Big(-i^2-i(I_{n-1}-n+e_{ss}+3-2k)+(k-1)(I_{n-1}-n+e_{ss}+2-k)
\Big)v^+\\
&\ \ + \sum_{l\neq s}^{n-1} (k-i-1)e_{ns}^{k-i-2}e_{nl}e_{ls}v^+
+e_{ns}^{k-i-1}y_nv^+,
\endaligned $$
and
$$\aligned
\  \sum_{l=1}^{n-1} e_{sl}e_{ls}e_{ns}^{k-i-1}v^+
& = e_{ss}e_{ss}e_{ns}^{k-i-1}v^+
+\sum_{l\neq s}^{n-1} e_{sl}e_{ls}e_{ns}^{k-i-1}v^+\\
& = e_{ns}^{k-i-1}(e_{ss}+i+1-k)^2v^+
-\sum_{l\neq s}^{n-1}(k-i-1) e_{ns}^{k-i-2}e_{nl}e_{ls}v^+\\
&\ \ +\sum_{l\neq s}^{n-1} e_{ns}^{k-i-1}e_{sl}e_{ls}v^+\\
& = e_{ns}^{k-i-1}\Big(i^2+2i(e_{ss}+1-k)+(e_{ss}+1-k)^2\Big)v^+\\
&\  -\sum_{l\neq s}^{n-1}(k-i-1) e_{ns}^{k-i-2}e_{nl}e_{ls}v^+
 +\sum_{l\neq s}^{n-1} e_{ns}^{k-i-1}e_{sl}e_{ls}v^+.
\endaligned $$
Then
$$\aligned \  y_{s}e_{ns}^{k}v^+
&=\sum_{i=0}^{k-1}e_{ns}^{i}
\Big(e_{ss}(I_{n-1}-n)
+\sum_{l=1}^{n-1} e_{sl}e_{ls}-y_n\Big)e_{ns}^{k-i-1}v^+\\
&=\sum_{i=0}^{k-1}e_{ns}^{k-1}\Big(i^2+i(e_{ss}+2-2k+I_{n-1}-n)
+(e_{ss}+1-k)(I_{n-1}-n+1-k)\Big)v^+\\
&\ \ +\sum_{i=0}^{k-1}e_{ns}^{k-1}\Big(i^2
+2i(e_{ss}+1-k)+(e_{ss}+1-k)^2\Big)v^+
+\sum_{l\neq s}^{n-1} k e_{ns}^{k-1}e_{sl}e_{ls}v^+\\\\
&\ \ +\sum_{i=0}^{k-1}e_{ns}^{k-1}\Big(i^2+i(I_{n-1}-n+e_{ss}+3-2k)
-(k-1)(I_{n-1}-n+e_{ss}+2-k)
\Big)v^+\\
&\ \ -k(k-1)\sum_{l\neq s}^{n-1}e_{ns}^{k-2}e_{nl}e_{ls}v_\l-ke_{ns}^{k-1}y_nv^+\\
&=ke_{ns}^{k-1}\Big( (e_{ss}+1-k)^2+(e_{ss}+1-k)(I_{n-1}-n)
+\sum_{l\neq s}^{n-1} e_{sl}e_{ls}-y_n\Big)v^+\\
&\ \ -k(k-1)\sum_{l\neq s}^{n-1}e_{ns}^{k-2}e_{nl}e_{ls}v^+.
\endaligned $$
\end{proof}

\subsection{ Finite-dimensional irreducible modules over $\sigma\tau(W)$}

 In this subsection, we will determine when $L_{\sigma\tau(W)}(\l_1,\dots,\l_{n-1},\eta_n)$ is finite-dimensional and the relation
between $L_{\sigma\tau(W)}(\l_1,\dots,\l_{n-1},\eta_n)$ and a finite-dimensional $\gl_n$-module.

\begin{proposition}\label{finite-condition}For $(\l_1,\dots,\l_{n-1},\eta_n)\in \C^{n}$, if the module $L_{\sigma\tau(W)}(\l_1,\dots,\l_{n-1},\eta_n)$ is finite-dimensional, then
\begin{equation}\label{dominant}
\l_i-\l_{i+1}\in \Z_{\geq 0}\qquad(1\leq i\leq n-2),
\end{equation}
and  there is  $k\in \Z_{\geq 1}$
such that
$$\eta_n=
(\l_{n-1}+1-k)(\l_1+\dots+\l_{n-1}-n+\l_{n-1}+1-k).$$
\end{proposition}
\begin{proof}
By the theory of finite-dimensional modules over $\gl_{n-1}$,
(\ref{dominant}) holds. Since the subspace $\text{span}\{e_{n,n-1}^iv^+| i\in \Z_{\geq 0}\}$ of
$L_{\sigma\tau(W)}(\l_1,\dots,\l_{n-1},\eta_n)$  is finite-dimensional, we can choose some $k\in \Z_{\geq 1}$ such that
$e_{n,n-1}^kv^+=0$ but $e_{n,n-1}^{k-1}v^+\neq 0$. Next we determine the
action of $y_{n-1}$ on $e_{n,n-1}^kv^+$.

By Lemma \ref{key-lemma} and $e_{l,n-1}v^+=0$ for $l<n-1$,
we obtain that
$$\aligned &y_{n-1}e_{n,n-1}^kv^+\\
=&ke_{n,n-1}^{k-1}\Big( (e_{n-1,n-1}+1-k)^2+(e_{n-1,n-1}+1-k)(I_{n-1}-n)\\
&\ \ +\sum_{l=1}^{n-2} e_{n-1,l}e_{l,n-1}-y_n\Big)v^+
 -k(k-1)\sum_{l=1}^{n-2}e_{n,n-1}^{k-2}e_{n,l}e_{l,n-1}v^+\\
=&  ke_{n,n-1}^{k-1}\Big((e_{n-1,n-1}+1-k)^2+(I_{n-1}-n)(e_{n-1,n-1}+1-k)
-y_n\Big)v^+=0.
\endaligned$$
Therefore $(\l_{n-1}+1-k)^2+(\l_1+\dots+\l_{n-1}-n)(\l_{n-1}+1-k)
-\eta_n=0$.
\end{proof}

We view every  $\lambda=(\l_1,\dots,\l_{n})\in \C^n$
as a  linear map $\lambda$ on $(\oplus_{i=1}^n \C e_{ii})$
such that $\lambda(e_{ii})=\lambda_i$ for any $i$.
For $\l\in \C^{n}$, let $V(\l)$
be the irreducible highest weight module over $\gl_n$ with the highest weight $\lambda$. Let $\Lambda_n^+=\{\l\in \C^n\mid \l_i-\l_{i+1}\in \Z_{\geq 0}\}$.
It is well known that  $V(\l)$ is finite-dimensional if and only if $\l\in \Lambda_n^+$,
see \cite [Theorem 21.2]{H}.

Note that $\sigma\tau(W)\subset U(\gl_n)$, see Lemma \ref{w-algebra-3.2}. Restricted to $\sigma\tau(W)$, any finite-dimensional irreducible $\gl_n$-module has a finite-dimensional irreducible  $\sigma\tau(W)$-quotient module.
We will show that the converse statement is still true.

\begin{theorem}\label{finite-1}If $M$ is a finite-dimensional irreducible  $\sigma\tau(W)$-module, then there is a finite-dimensional irreducible $\gl_n$-module $V(\l)$ such that $M$ is isomorphic to an irreducible $\sigma\tau(W)$-quotient module
of  $V(\l)$.
\end{theorem}
\begin{proof}
We first explain  that $M$ is a highest weight module. The  action of
$I_{n-1}$ on $M$ has a finite number of eigenvalues.
Let $K\subset M$ be the eigenspace of $I_{n-1}$ corresponding to
eigenvalue with the maximal real part.
Since  $[I_{n-1},y_j]=y_j$, the maximality implies that $y_j K=0$
for all $j<n$.
 The relations in (1) of 
Lemma~\ref{lemma-3.3} show that $K$ is stable under $\gl_{n-1}$ and $y_n$. As a
finite-dimensional $\gl_{n-1}$-module, $K$ is completely reducible. Thus  $K$ must contain
a $\gl_{n-1}$-highest weight vector. By
fact that $[\gl_{n-1},y_n]=0$, we can choose a common eigenvector $v^+\in K$
of $e_{11},\dots,e_{n-1,n-1}, y_n$ such that $e_{ij}v^+=0$
for all $1\leq i<j<n$.
Therefore
\[
 M\cong L_{\sigma\tau(W)}(\l_1,\dots,\l_{n-1},\eta_n)
\]
for suitable $\l_1,\dots,\l_{n-1},\eta_n\in\C$.

By Proposition \ref{finite-condition},
$$\l_1-\l_2,\dots,\l_{n-2}-\l_{n-1}\in \Z_{\geq 0},$$
and there is  $k\in \Z_{\geq 1}$
such that
\begin{equation}\label{equation-l}\eta_n=
(\l_{n-1}+1-k)(\l_1+\dots+\l_{n-1}-n+\l_{n-1}+1-k).\end{equation}

Let $\l_n:=\l_{n-1}+1-k$. Then $\l_{n-1}-\l_n=k-1\in \Z_{\geq 0} $ and
$\eta_n=\l_n(\l_1+\dots+\l_{n-1}+\l_n-n)$.
This means that the  irreducible highest weight $\gl_n$-module $V(\l_1,\dots,\l_{n})$ is finite-dimensional. Restricted to $\sigma\tau(W)$,
the highest weight vector $v_\l$ of $V(\l_1,\dots,\l_{n})$ is annihilated by
$\sigma\tau(W)^+$, and
$$y_nv_{\l}=
(e_{nn}(I_{n}-n)
+\sum_{j=1}^{n-1} e_{nj}e_{jn} ) v_{\l}=
\eta_n v_{\l}.$$
Therefore $V(\l_1,\dots,\l_{n})$
 has a quotient module  isomorphic to $L_{\sigma\tau(W)}(\l_1,\dots,\l_{n-1},\eta_n)$.
\end{proof}

\begin{corollary}For $(\l_1,\dots,\l_{n-1},\eta_n)\in \C^{n}$,   $L_{\sigma\tau(W)}(\l_1,\dots,\l_{n-1},\eta_n)$ is finite-dimensional if and only if
$$\l_1-\l_2,\dots,\l_{n-2}-\l_{n-1}\in \Z_{\geq 0},$$
and  there is  $\l_n\in \C$
such that
$$
\l_{n-1}-\l_n\in \Z_{\geq 0}, \quad
\eta_n=
\l_n(\l_1+\dots+\l_{n-1}+\l_n-n).$$
\end{corollary}

For $\lambda=(\lambda_1,\ldots,\lambda_n)\in\Lambda_n^+$, let
$V(\lambda)$ be the finite-dimensional irreducible $U(\mathfrak{gl}_n)$-module
with highest weight  $\lambda$.
Denote $$I(\lambda)=\{\mu=(\mu_1,\dots\mu_{n-1})\in \C^{n-1}\mid\l_1\geq \mu_1\geq \l_2\geq \mu_2\geq\cdots \l_{n-1}\geq\mu_{n-1}\geq \l_n\},$$
where $a\geq b$ implies that $a-b\in \Z_{\geq 0}$.
If $\mu\in I(\lambda)$, we say that $\mu$
interlaces $\l$. Then by the branching rule \cite[Theorem 8.1.2]{GW}, we have the following multiplicity-free decomposition:
\begin{equation}\label{multiplicity-free}V(\l)=\oplus_{\mu\in I(\l)}V_{n-1}(\mu),\end{equation}
where $V_{n-1}(\mu)$ is the finite-dimensional irreducible
$\gl_{n-1}$-module of the highest weight $\mu$.

\begin{theorem}\label{sub-mod}
Let 
$V(\lambda)$ be the finite-dimensional irreducible $U(\mathfrak{gl}_n)$-module
with highest weight $\lambda\in\Lambda_n^+$. Suppose that there is
$s\in\{1,\ldots,n-1\}$ such that
\begin{equation}\label{red-condition}
 k:=\lambda_s-s+|\lambda|\in\mathbb Z_{\geq 1},
 \qquad
 s-|\lambda|-\lambda_{s+1}\in\mathbb Z_{\geq 0}.
\end{equation}
Set
\[
 \nu=(\lambda_1,\ldots,\lambda_{s-1},s-|\lambda|,
       \lambda_{s+1},\ldots,\lambda_{n-1}).
\]
Then $\nu$ interlaces $\lambda$. Let $v_\nu$ be a
$\mathfrak{gl}_{n-1}$-highest weight vector in the unique summand
$V_{n-1}(\nu)$ of the branching decomposition
(\ref{multiplicity-free}).
Then
\[
 \sigma\tau(W)^+v_\nu=0.
\]
Consequently, $v_\nu$ generates a proper $\sigma\tau(W)$-submodule of
$V(\lambda)$. In particular, the restriction of $V(\lambda)$ to
$\sigma\tau(W)$ is reducible.
\end{theorem}

\begin{proof}
By (\ref{red-condition}), we have
\[
 \lambda_s-\nu_s=k\in\mathbb Z_{\geq1},
 \qquad
 \nu_s-\lambda_{s+1}\in\mathbb Z_{\geq0}.
\]
 Thus $\nu$ interlaces $\lambda$. Hence the branching rule gives a unique
summand $V_{n-1}(\nu)$.

For $m\geq1$, put
\[
 \Omega_m:=\sum_{1\leq i,j\leq m}e_{ij}e_{ji},
 \qquad
 c_m(\alpha):=\sum_{j=1}^m
 \alpha_j\bigl(\alpha_j+m+1-2j\bigr).
\]
Thus $\Omega_m$ acts on the irreducible highest weight
$\mathfrak{gl}_m$-module of highest weight $\alpha$ by the scalar
$c_m(\alpha)$. Define
\[
 D:=y_n+\frac12e_{nn}(e_{nn}-1).
\]
Since $y_n$ and $e_{nn}$ commute with $\mathfrak{gl}_{n-1}$, so does $D$.
From
\[
 [e_{in},y_n]=y_i-e_{nn}e_{in},
 \qquad
 \left[e_{in},\frac12e_{nn}(e_{nn}-1)\right]=e_{nn}e_{in},
\]
it follows that
\[
 y_i=[e_{in},D],\qquad 1\leq i\leq n-1.
\]

Using
\[
 \sum_{j=1}^{n-1}e_{nj}e_{jn}
 =\frac12\bigl(\Omega_n-\Omega_{n-1}-e_{nn}^2-I_{n-1}
 +(n-1)e_{nn}\bigr),
\]
we obtain
\[
 D=\frac12(\Omega_n-\Omega_{n-1})
   +\left(I_n-\frac{n+1}{2}\right)e_{nn}-\frac12I_n.
\]
Therefore $D$ acts on the summand $V_{n-1}(\mu)\subset V(\lambda)$ by the scalar 
\[
 d_\lambda(\mu)
 =\frac12\bigl(c_n(\lambda)-c_{n-1}(\mu)\bigr)
 +\left(|\lambda|-\frac{n+1}{2}\right)(|\lambda|-|\mu|)
 -\frac12|\lambda|.
\]
Whenever $\mu+\mathbf e_i$ interlaces $\lambda$, one has
\[
 d_\lambda(\mu+\mathbf e_i)-d_\lambda(\mu)
 =i-|\lambda|-\mu_i. \tag{*}
\]

We first show inductively that $e_{in}v_\nu=0$ for $1\leq i<s$.
If $e_{in}v_\nu\neq0$ and $e_{an}v_\nu=0$ for all $a<i$, then
$e_{in}v_\nu$ is a $\mathfrak{gl}_{n-1}$-highest weight vector of weight
$\nu+\mathbf e_i$, because
\[
 [e_{ab},e_{in}]=\delta_{bi}e_{an}\qquad(a<b<n).
\]
But $\nu_i=\lambda_i$, so $\nu+\mathbf e_i$ does not interlace $\lambda$,
a contradiction. Hence $e_{in}v_\nu=0$ for all $i<s$. It follows that
$e_{sn}v_\nu$ is either zero or a $\mathfrak{gl}_{n-1}$-highest weight
vector of weight $\nu+\mathbf e_s$. In the latter case it lies in the
unique summand $V_{n-1}(\nu+\mathbf e_s)$. Thus, by $(*)$, we have
\[
 \begin{aligned}
 y_sv_\nu
 &= [e_{sn},D]v_\nu\\
 &=\bigl(d_\lambda(\nu)-d_\lambda(\nu+\mathbf e_s)\bigr)e_{sn}v_\nu\\
 &=\bigl(\nu_s+|\lambda|-s\bigr)e_{sn}v_\nu=0,
 \end{aligned}
\]
because $\nu_s=s-|\lambda|$.

We now prove that $y_iv_\nu=0$ for every $i\neq s$. For
$i=1,\ldots,s-1$, proceed by induction. If $y_av_\nu=0$ for all $a<i$,
then Lemma~3.3 gives that $y_iv_\nu$ is a
$\mathfrak{gl}_{n-1}$-highest weight vector of weight
$\nu+\mathbf e_i$. Since $\nu_i=\lambda_i$, this weight does not interlace
$\lambda$, so $y_iv_\nu=0$. Starting from $y_sv_\nu=0$, the same argument
applies successively to $i=s+1,\ldots,n-1$, because again
$\nu_i=\lambda_i$. Hence $y_iv_\nu=0$ for all $1\leq i\leq n-1$.
Together with $e_{ij}v_\nu=0$ for $1\leq i<j\leq n-1$, this proves
$\sigma\tau(W)^+v_\nu=0$.

Finally,
$\sigma\tau(W)v_\nu=\sigma\tau(W)^-\sigma\tau(W)^0v_\nu$. The generators
$e_{ni}$ decrease the $I_{n-1}$-weight by one, while the remaining
generators of $\sigma\tau(W)^-$ preserve it. Thus every vector in
$\sigma\tau(W)v_\nu$ has $I_{n-1}$-weight at most $|\nu|$. On the other
hand,
\[
 (\lambda_1+\cdots+\lambda_{n-1})-|\nu|
 =\lambda_s-s+|\lambda|=k>0.
\]
Therefore $v_\lambda\notin\sigma\tau(W)v_\nu$, and 
submodule $\sigma\tau(W)v_\nu$ is proper.
\end{proof}


In fact, the condition (\ref{red-condition}) is the necessary and sufficient condition for the reducibility of  $V(\lambda)$ as a $\sigma\tau(W)$-module.
To show that  necessity, we need the  generalized central character of $V(\lambda)$ as a $\sigma\tau(W)$-module, see Lemmas \ref{centra-character}, \ref{centra-class} and Theorem \ref{irre-factor}.

Let $A_{n}= \C [x_1, \dots, x_{n}]$.
The subalgebra of $\text{End}_{\C}(A_{n})$  generated by $\{x_i,\frac{\partial}{\partial x_i}\mid 1\leq i\leq n\}  $ is the Weyl algebra $D_{n}$ of rank $n$. From the construction of Shen-Larsson module over  $W^+_n=\mathrm{Der}(A_n)$ and an embedding from
$\sl_{n+1}$ to $W_n^+$,  see \cite{Sh,LLZ}, we have the following algebra homomorphism.

\begin{proposition}\label{s-iso} For any  $n\in \Z_{\geq 2}$, there is an algebra homomorphism
$$\psi: U(\sl_{n+1}) \rightarrow D_{n}\otimes U(\gl_n)$$
defined by
\begin{equation}\label{3.1}\aligned
\ h_k& \mapsto x_k\frac{\partial}{\partial x_k}\otimes 1+1\otimes e_{kk},
\ e_{ij} \mapsto 1\otimes e_{ij}+x_i\frac{\partial}{\partial x_j}\otimes 1,\\
\ e_{0j}&\mapsto -\frac{\partial}{\partial x_j}\otimes 1,
\ e_{i0}\mapsto  \sum_{q=1}^nx_q\otimes e_{iq}+
x_i\sum_{q=1}^nx_q\frac{\partial}{\partial x_q}\otimes 1+ x_i\otimes I_n,
\endaligned \end{equation}   where $i,j,k \in \{1,\dots,n\}$ with $i\neq j$, and $I_n:=\sum_{i=1}^n e_{ii}$ is the identity matrix in $\gl_n$.
\end{proposition}

For a  $\gl_n$-module $V$ and a  module $P$ over $D_n$,  by $T(P, V ) $ we denote the space $P \otimes  V$
considered as a module over $ \sl_{n+1}$ through the homomorphism $\psi$. Let $\sigma_{\b1}$ be the algebra automorphism of $D_n$ defined by $$x_i\mapsto x_i, \ \ \frac{\partial}{\partial x_i} \mapsto \frac{\partial}{\partial x_i}-1.$$
The $D_n$-module $A_n$ can be twisted by $\sigma_{\b1}$ to be a new $D_n$-module $A_n^{\b1}$ (see $\S 2.9$ in \cite{GN}) whose module structure is defined by
\begin{equation}\label{A-mod}X\cdot f(x_1,\dots,x_n)=\sigma_{\b1}(X)f(x_1,\dots,x_n), \ \ f(x_1,\dots,x_n)\in A_n, X\in D_n.\end{equation}

The following lemma is helpful for discussing the action of $Z(\tau(W))$ on $V(\l)$.

\begin{lemma}Let $\l\in \Lambda_n^+$. Then $B_{\b1}\otimes V(\l)^\tau\cong T(A_n^{\b1}, V(\l))$ as $\sl_{n+1}$-modules, where $B_{\b1}$ is the irreducible $B$-module defined in the proof of Theorem \ref{iso-thw}, and the action of $\sl_{n+1}$ on $B_{\b1}\otimes V(\l)^\tau$ is defined by the isomorphism $U_S\cong B\otimes W$.
\end{lemma}
\begin{proof}Set
$$\mathrm{wh}_\b1(T(A_n^{\b1}, V(\l)))=\{v\in T(A_n^{\b1}, V(\l))\mid e_{0i}v=v,i=1,\dots,n\}.$$ In $T(A_n^{\b1}, V(\l))$, for $v\in V(\lambda)$,
$$e_{0j}\cdot (1\otimes v)=-\big((\frac{\partial}{\partial x_i}-1)\otimes 1\big)(1\otimes v)=1\otimes v.$$
Conversely, the simultaneous equations $(e_{0j}-1)(f\otimes v)=0$ force the
polynomial part $f$ to be constant. Hence
$\mathrm{wh}_{\b1}(T(A_n^{\b1},V(\l)))=1\otimes V(\l)$. Since $[W,e_{0i}]=0$,  $W\mathrm{wh}_\b1(T(A_n^{\b1}, V(\l)))\subset \text{wh}_\b1(T(A_n^{\b1}, V(\l)))$, i.e., $W V(\l)\subset V(\l)$.
According to the definition of $\psi$ in (\ref{3.1}), by a direct computation, the action of $W$ on $\text{wh}_{\b1}T(A_n^\b1,V)=V(\l)$ is described as follows:
\begin{equation}\label{x-w-act2}\aligned x_{ij}v &=  (e_{ij}-e_{ii})v,\ 1\leq i\neq j\leq n,\\
 \omega_k v
&=\sum_{j=1}^n (e_{kj}e_{jj}-e_{kj})v,\ 1\leq k\leq n,
\endaligned \end{equation}where $ v\in V(\l)$. So
as a $W$-module, $\mathrm{wh}_\b1(T(A_n^{\b1}, V(\l)))\cong V(\l)^\tau$.
 By Lemma 6 in \cite{LL}, $T(A_n^{\b1}, V(\l))=U(\mh_n)\otimes V(\l)\cong B_{\b1}\otimes V(\l)^\tau$.
\end{proof}

For $\l=(\l_1,\dots,\l_n)\in \Lambda_n^+$, set $\l_0=-|\l|$ and $\widetilde{\l}=(\l_0,\l_1,\dots,\l_n)$. Since restricted to $\sigma\tau(W)$, $V(\l)$ is also a highest weight module,  it follows that there is a character $\chi: Z(\sigma\tau(W))\rightarrow \C$ such that $z-\chi(z)$ acts locally nilpotently on $V(\l)$ for any $z\in Z(\sigma\tau(W))$. So as a $\sigma\tau(W)$-module, $V(\l)$ has a generalized
central character.

\begin{lemma}\label{centra-character}  Let $\l,\l'\in \Lambda_n^+$. The following statements are equivalent:
\begin{enumerate}
\item With respect to the action of $Z(\sigma\tau(W))$,  $V(\l)$ and $V(\l')$ have the same  generalized central character.
\item With respect to  the action of  $Z(\tau(W))$,  $V(\l)$ and $V(\l')$ have the same  generalized central character.
\item With respect to the action of $Z(\sl_{n+1})$, the $\sl_{n+1}$-modules $T(A_n, V(\l))$ and $T(A_n, V(\l'))$ have the same  generalized central character.
\item $\widetilde{\l}$ and $\widetilde{\l'}$ are conjugate under the dot action of the symmetric group $S_{n+1}$.
\end{enumerate}
\end{lemma}
\begin{proof}
 Since $\sigma$ is an automorphism, it identifies the centers of $\tau(W)$
and $\sigma\tau(W)$; hence (1) and (2) are equivalent. By
Theorem~\ref{main-iso-th}(a), the $\tau(W)$-modules $V(\l)$ and $V(\l')$
have the same generalized central character if and only if the
$\sl_{n+1}$-modules $B_{\b1}\otimes V(\l)^\tau$ and
$B_{\b1}\otimes V(\l')^\tau$ do. The preceding lemma identifies these
modules with $T(A_n^{\b1},V(\l))$ and $T(A_n^{\b1},V(\l'))$, respectively.
Twisting the polynomial Weyl-algebra factor does not change the induced
$Z(\sl_{n+1})$-character, so (2) and (3) are equivalent.

Since $T(A_n, V(\l))$ is indecomposable  and its $\sl_{n+1}$-submodule generated by the highest weight vector of $V(\l)$ is a  highest weight module, or by Proposition 3.4 in \cite{CG},  the $\sl_{n+1}$-module $T(A_n, V(\l))$  has the generalized central character $\chi_{\widetilde{\l}}$ decided by the highest weight  $\widetilde{\l}$.
Then by the Harish-Chandra's Theorem on central characters,  (3) and (4) are equivalent.
\end{proof}

By Lemma \ref{centra-character}, we can define an equivalence relation
on $\Lambda_n^+$ such that:  $\l \sim \l'$ if and only if $\widetilde{\l}$ and $\widetilde{\l'}$ are conjugate under the dot action of the symmetric group $S_{n+1}$. For $\l\in\Lambda_n^+$ , denote the equivalence
class containing $\l$ by $[\l]$. The following lemma was known from page 565 in \cite{M}, or see Lemma 2.1 in \cite{CG}.

\begin{lemma}\label{centra-class}  Let $\l\in \Lambda_n^+$.
\begin{enumerate}
\item If $-|\l|-\l_1\not\in \Z$, then $[\l]=\{\l\}$.
\item If $-|\l|=\l_i-i$ for some $i\in\{1,\dots,n\}$, then $[\l]=\{\l\}$.
\item If $-|\l|-\l_1\in \Z$ and $-|\l|,\l_1-1,\dots, \l_n-n$ are distinct, then there is a unique $\mu\in \Lambda_n^+$ such that $$-|\mu|-\mu_1\in \Z_{\geq 0}\ \text{and}\
[\l]=\{\mu,\mu(1),\dots,\mu(n)\},$$
where $\mu(i)=(\mu_0+1,\dots,\mu_{i-1}+1,\mu_{i+1},\dots,\mu_n)\in \Lambda_n^+$, and $\mu_0=-|\mu|$.
\end{enumerate}
\end{lemma}

\begin{proof} For the convenience of readers, we provide a proof of this lemma.
 Let $\l'\in [\l]$. Since
$\widetilde{\l}$ and $\widetilde{\l'}$ are conjugate under the dot action of  $S_{n+1}$, it follows that there is a permutation $\omega\in S_{n+1}$
such that $\omega(\widetilde{\l}+\rho)=\widetilde{\l'}+\rho$, i.e.,
\begin{equation}\label{permutation}\omega(\l_0,\dots,\l_i-i,\dots,\l_n-n)
=(\l'_0,\dots,\l'_i-i,\dots,\l'_n-n),\end{equation}
where $\l_0=-|\l|$. Note that $\widetilde{\l}+\rho
=(\frac{1}{2}n+\l_0,\dots,\frac{1}{2}n+\l_i-i,\dots,\frac{1}{2}n+\l_n-n)$.

(1) Since $-|\l|-\l_1\not\in \Z$ and $\l\in \Lambda_n^+$, it follows that $\l_0-\l_i\not\in \Z$ for any $i\in \{1,\dots,n\}$, and $$\l_k-k-(\l_{k+1}-k-1),\l'_k-k-(\l'_{k+1}-k-1)\in \Z_{\geq 1},$$ for any $k\in \{1,\dots,n-1\}$. By checking   two sides of equation (\ref{permutation}),
 we deduce that $\omega=\mathrm{id}$ and $\l'=\l$.

(2) If $-|\l|=\l_i-i$ for some $i\in\{1,\dots,n\}$, then   $\omega$ is the transposition $(0i)$, by comparing   two sides of equation (\ref{permutation}). Hence
$\l'=\l$.

(3) Suppose that $\l_i-i>-|\l|>\l_{i+1}-i-1$ for some $i\in\{1,\dots,n\}$, where $a>b$ means $a-b\in \Z_{>0}$. We can choose  $\mu\in \Lambda_n^+$ such $(0\cdots i)\cdot \widetilde{\mu}= \widetilde{\l}$.
Indeed, from $(0\cdots i)\cdot \widetilde{\mu}=
(\mu_i-i,\mu_0+1,\dots,\mu_{i-1}+1,\mu_{i+1},\dots,\mu_n)$,
we have
$\mu_0=\l_1-1,\dots,\mu_{i-1}=\l_i-1,\mu_i=i-|\l|, \mu_{j}=\l_j$ for $j\geq i+1$.
Then $-|\mu|-\mu_1=\mu_0-\mu_1=\l_1-\l_2\in \Z_{\geq 0}$ and
$[\l]=[\mu]=\{\mu,\mu(1),\dots,\mu(n)\}$.
\end{proof}

%
%

Now we are in  a position to
decide the irreducibility of  $V(\l)$
when it is restricted to $\sigma\tau(W)$.

\begin{theorem}\label{irre-factor}  Let $\l\in \Lambda_n^+$.
\begin{enumerate}
\item If $-|\l|-\l_1\not\in \Z$, then restricted to $\sigma\tau(W)$, $V(\l)$ is irreducible.
\item  If $-|\l|=\l_i-i$ for some $i\in\{1,\dots,n\}$, then restricted to $\sigma\tau(W)$, $V(\l)$ is irreducible.
\item Let $\mu\in \Lambda_n^+$ such that $-|\mu|-\mu_1\in \Z_{\geq 0}$.
 Restricted to $\sigma\tau(W)$,  the composition length of $V(\mu(i))$ is $2$ for $i=1,\dots,n-1$, but $V(\mu)$  and $V(\mu(n))$ are irreducible.
\end{enumerate}
\end{theorem}

\begin{proof} Write $A=\sigma\tau(W)$. Let $N$ be a simple composition factor of
$V(\l)|_A$. By the highest weight module theory of $A$,
\[
 N\cong L_A(\beta_1,\dots,\beta_{n-1},\eta),
\]
where $\beta=(\beta_1,\dots,\beta_{n-1})$ interlaces $\l$. By
Theorem~\ref{finite-1}, there is a weight $\nu\in\Lambda_n^+$ with
\[
 \nu=(\beta_1,\dots,\beta_{n-1},\nu_n)
\]
such that $N$ is
a quotient of $V(\nu)|_A$. Since $N$ is also a subquotient of $V(\l)|_A$,
the  modules $V(\l)|_A$ and $V(\nu)|_A$ have the same generalized central character. Hence
$\nu\in[\l]$ by Lemma~\ref{centra-character}.

Suppose first that $[\l]=\{\l\}$. Then necessarily $\nu=\l$. The
corresponding $\gl_{n-1}$-highest weight occurs only once in the
multiplicity-free branching decomposition \eqref{multiplicity-free}. Thus
$V(\l)|_A$ has a single simple composition factor, with multiplicity one.
It is therefore irreducible. Lemma~\ref{centra-class}(1) and (2) give
assertions (1) and (2).

It remains to consider the class
$[\mu]=\{\mu,\mu(1),\dots,\mu(n)\}$. For convenience set
$\mu(0)=\mu$. Apply the preceding paragraph to $V(\mu(i))$. The possible
$\nu$ are found by requiring that
$(\nu_1,\dots,\nu_{n-1})$ interlace $\mu(i)$ and that the $y_n$-eigenvalue
agree. Using the scalar $d_\lambda(\beta)$ computed in the proof of
Theorem~\ref{sub-mod}, a direct substitution gives
\[
 \begin{array}{c|c}
 i & \text{possible }\nu\\ \hline
 0 & \mu(0)\\
 1\leq i\leq n-1 & \mu(i-1),\ \mu(i)\\
 n & \mu(n).
 \end{array}
\]
Each possibility has multiplicity one by the branching rule. Thus the
composition length is at most two for $1\leq i\leq n-1$ and at most one
at the two endpoints. For an interior index $i$, the hypotheses of
Theorem~\ref{sub-mod} hold with $s=i$, because
\[
 \lambda_i-i+|\lambda|=\mu_{i-1}-\mu_i+1\geq1,
 \qquad
 i-|\lambda|-\lambda_{i+1}=\mu_i-\mu_{i+1}\geq0,
\]
where $\lambda=\mu(i)$. Hence $V(\mu(i))|_A$ is reducible and has length
exactly two. The endpoint modules $V(\mu)$  and $V(\mu(n))$ have length one and are irreducible.
\end{proof}

The following example gives irreducible $\sigma\tau(W)$ sub-quotients of finite-dimensional irreducible $\gl_n$-modules $V(\delta_{k})$  with fundamental weights $\delta_k=e_1+\dots+e_k$.

\begin{example} For  $\gl_n$-modules $V(\delta_{k}):=\wedge^k \C^n$, $k=1,\dots,n$,  we have the following exact sequence of $\sigma\tau(W)$-modules:
\begin{equation}\label{complex}\xymatrix
{0\ar[r] & V(\delta_0)\ar[r]^{\pi_0} &V(\delta_1)\ar[r]^{\hskip 0.1cm\pi_1} &\cdots
\cdots\ar[r]^{\pi_{n-2}\hskip 0.1cm}&{V(\delta_{n-1})}\ar[r]^{\hskip 0.1cm\pi_{n-1}} &V(\delta_n)\ar[r]& 0,}
\end{equation}
where
$\pi_{k-1}$ is the $\sigma\tau(W)$-module homomorphism defined by
\begin{equation*}\begin{array}{lrcl}
\pi_{k-1}:& V(\delta_{k-1}) & \rightarrow & V(\delta_{k}),\\
      & v & \mapsto & e_n\wedge v,
\end{array}\end{equation*}
for all $v\in V(\delta_{k-1})$, $k\in \{1,\cdots, n\}$, $V(\delta_0)=\C$ is the one-dimensional trivial $\gl_n$-module, and
$e_n\wedge a:=ae_n$, for $a\in \C$.
Moreover
$$\mathrm{im}\pi_{k-1}=\mathrm{span}\{e_{i_1}\wedge\dots \wedge e_{i_{k-1}}\wedge e_n\mid
1\leq i_1<i_2<\dots<i_{k-1}<n\}$$ is an irreducible $\sigma\tau(W)$-module.
\end{example}

\subsection{ Finite-dimensional irreducible modules over $W$}

By the homomorphism $\tau$ defined by (\ref{x-w-act}), any $\gl_n$-module $V$
can be defined to be a $W$-module denoted by $V^\tau$
such that $x\cdot v=\tau(x)v$ for all $x\in W,v\in V$.

\begin{theorem}\label{finite-w-mod}If $M$ is a finite-dimensional  irreducible $W$-module, then there is a finite-dimensional
irreducible $\gl_n$-module $V$ such that $M$ is isomorphic to an irreducible $W$-quotient module of $V^\tau$.
\end{theorem}

\begin{proof} We can view $M$ as a $\sigma\tau(W)$-module still denoted by $M$ in the following way:
$$\sigma\tau(u) \circ v= u v,\  \forall\  u\in W, v\in M.$$
By Theorem \ref{finite-1}, there is a finite-dimensional irreducible $\gl_n$-module $V(\l)$ such that  the $\sigma\tau(W)$-module $M$ is isomorphic to an irreducible $\sigma\tau(W)$-quotient module
of $V(\l)$.
Recall that $\sigma=\textup{exp}(-\textup{ad} X)$, and $\sigma$ maps the positive part of $\gl_n$  to the positive part, and $\sigma(h_n)=h_n-X, \sigma(h_i)=h_i+e_{in}$ for $1\leq i\leq n-1$. Thus for $\l\in \Lambda_n^+$, the new $\gl_n$-module $V(\l)^\sigma$ twisted
from $V(\l)$  by $\sigma$ is also isomorphic to  $V(\l)$. Similarly we also have $V(\l)^{\sigma^{-1}} \cong V(\l)$.
Then we can complete the proof.
\end{proof}

Since $V(\l)^\sigma \cong V(\l)$ and $V(\l)^{\sigma^{-1}} \cong V(\l)$,
restricted  to the subalgebra $\tau(W)$, the statements on $V(\l)$ in Theorem \ref{irre-factor} also hold. Therefore we have the following result on the
$W$-module $V(\l)^\tau$ for any finite-dimensional
irreducible $\gl_n$-module $V(\l)$.

\begin{theorem}\label{irre-factor2}  Let $\l\in \Lambda_n^+$.
\begin{enumerate}
\item If $-|\l|-\l_1\not\in \Z$, then the $W$-module $V(\l)^\tau$ is irreducible.
\item  If $-|\l|=\l_i-i$ for some $i\in\{1,\dots,n\}$, then the $W$-module $V(\l)^\tau$ is irreducible.
\item Let $\mu\in \Lambda_n^+$ such that $-|\mu|-\mu_1\in \Z_{\geq 0}$.
 Then the composition length of the $W$-module $V(\mu(i))^\tau$ is $2$ for $i=1,\dots,n-1$, but $V(\mu)^\tau$  and $V(\mu(n))^\tau$ are irreducible, where $\mu(i)=(\mu_0+1,\dots,\mu_{i-1}+1,\mu_{i+1},\dots,\mu_n)\in \Lambda_n^+$, and $\mu_0=-|\mu|$.
\end{enumerate}
\end{theorem}

\begin{remark} It is not easy to  obtain Proposition  \ref{finite-condition} from
Theorem 7.9 in \cite{BK2} and the arguments in the case of the Yangian $Y_n$ due to Tarasov \cite{Ta} and Drinfeld \cite{Dr}. The advantage of our method is more explicit,  avoiding using
representations over shifted Yangians. Moreover, we can determine the irreducibility  of  any finite-dimensional irreducible $\gl_n$-module $V(\l)$  when it was restricted to   $\sigma\tau(W)$ in Theorem \ref{irre-factor}.
\end{remark}

\section{Realizations of  irreducible cuspidal modules over $\sl_{n+1}$}\label{Sec.4}

In this section, as an application, we give  realizations of
all irreducible cuspidal modules over $\sl_{n+1}$ using
$\gl_n$-modules and the induced module technique.

\subsection{Explicit realizations of cuspidal modules}

For  $\mu\in \C^n$ and a finite-dimensional $W$-module $N$, let each $h_{i}$
 act  on it as the scalar $\mu_i$.
  Define
  $$G_\mu(N)=\text{Ind}_{U(\mh_n)W}^{U_S}N=U_{S}\otimes_{U(\mh_n)W} N.$$ It is clear that $G_\mu(N)=\C[e_{01}^{\pm 1},\dots, e_{0n}^{\pm 1}]\otimes N$ by the isomorphism in Theorem \ref{main-iso-th}.
  From  (\ref{x-w-def}),
   the $\sl_{n+1}$-module structure on $G_\mu(N)$ satisfies:
  $$\aligned h_k \cdot (e^{\br}\otimes v)& =(\mu_k-r_k)e^{\br}\otimes v,\\
  e_{0k} \cdot (e^{\br}\otimes v)& =e^{\br+e_k}\otimes v,\\
   e_{ij}\cdot  (e^{\br}\otimes v)& =e^{\br-e_i+e_j}\otimes x_{ij} v+(\mu_i-r_i+1)e^{\br-e_i+e_j}\otimes v,\ i\neq j,\\
  e_{i0}\cdot (e^{\br}\otimes v)
  &= e^{\br-e_i}\otimes\Big( \omega_i v-\sum_{j=1}^n(\mu_j-r_j) x_{ij}v- (|\mu|-|\br|)(\mu_i-r_i+1) v\Big),
  \endaligned$$
where $v\in N, e^{\br}=e_{01}^{r_1}\dots e_{0n}^{r_n}, \br=(r_1,\dots, r_n)\in \Z^n, 1\leq i, j, k\leq n$, and $x_{ii}:=0$.

In particular if $V$ is a $\gl_n$-module, then
$G_\mu(V^\tau)$ is an $\sl_{n+1}$-module such that:
\begin{equation}\label{ind-mod}\aligned
   e_{ij}\cdot  (e^{\br}\otimes v)& =e^{\br-e_i+e_j}\otimes(e_{ij}-e_{ii}+\mu_i-r_i+1)v,\ i\neq j,\\
  e_{i0}\cdot (e^{\br}\otimes v)
  &= e^{\br-e_i}\otimes\Big( \sum_{j=1}^n (e_{ij}e_{jj}-e_{ij}) v-\sum_{j=1}^n(\mu_j-r_j) (e_{ij}-e_{ii})v\\
   & \ \ - (|\mu|-|\br|)(\mu_i-r_i+1) v\Big).
  \endaligned\end{equation}

\begin{lemma}\label{cuspidal-lemm}Suppose that $V(\l)$ is a finite-dimensional
irreducible $\gl_n$-module.
Then for $\mu\in \C^n$, $G_{\mu}(V(\l)^\tau)$ is a cuspidal $\sl_{n+1}$-module  if and only if
\begin{equation}\label{cus-condition}\mu_i-\l_i\not\in \Z,\ |\mu|+\l_i\not\in \Z,\ \text{for any}\
i\in \{1,\dots,n\}.
\end{equation}
\end{lemma}

\begin{proof}
 It is easy to see that $e_{0i}$ acts injectively on $G_\mu(V(\lambda)^\tau)$. Since  $V(\lambda)$ is a weight module, $V(\lambda)=\oplus_{\alpha\in \C^n}V(\lambda)_\alpha$, where $V(\lambda)_\alpha:=\{w \in V(\lambda) \mid e_{kk}w=\alpha_kw, k=1\dots, n\}$.
 For any nonzero $v=\sum_{\alpha \in I }v_{\alpha}\in V(\lambda)$, where $I\subset \C^n$ is a finite subset, $0\neq v_\alpha\in V(\lambda)_\alpha$, we have that
 $$\aligned
& (e_{ij}-e_{ii}+\mu_i-r_i+1)v
&=\sum_{\alpha \in I }e_{ij}v_{\alpha}
+ \sum_{\alpha \in I}(-\alpha_i+\mu_i-r_i+1)v_{\alpha}.
\endaligned$$
  Then by the first formula in (\ref{ind-mod}), for $j\neq i$, we can see that $e_{ij}\cdot (e^{\br}\otimes v)\neq 0$  if  $\mu_i-\l_i\not\in \Z$, since $\l_i-\alpha_i\in \Z$.
  Conversely, if $G_\mu(V(\lambda)^\tau)$ is cuspidal,
  choose $\alpha\in \mathrm{supp}V(\l)$ such that
  $\alpha_i$ is maximal, then from $(e_{ij}-e_{ii}+\mu_i-r_i+1)v_\alpha
=(-\alpha_i+\mu_i-r_i+1)v_{\alpha}\neq 0$, we have $\mu_i-\l_i\not\in \Z$.
  So   $e_{ij}$
acts injectively on $G_\mu(V(\lambda))$ if and only if  $\mu_i-\l_i\not\in \Z$.

 Furthermore,
 $$\aligned &  \sum_{j=1}^n (e_{ij}e_{jj}-e_{ij}) v-\sum_{j=1}^n(\mu_j-r_j) (e_{ij}-e_{ii})v- (|\mu|-|\br|)(\mu_i-r_i+1) v\\
 &=\sum_{\alpha \in I}\sum_{j=1}^n e_{ij}(\alpha_{j}-1) v_\alpha-\sum_{\alpha \in I}\sum_{j=1}^n(\mu_j-r_j) (e_{ij}-\alpha_{i})v_\alpha
 - \sum_{\alpha \in I}(|\mu|-|\br|)(\mu_i-r_i+1) v_\alpha\\
  &=\sum_{\alpha \in I}\sum_{j=1,j\neq i}^n e_{ij}(\alpha_{j}-1-\mu_j+r_j) v_\alpha
- \sum_{\alpha \in I}(|\mu|-|\br|+\alpha_i)(-\alpha_{i}+\mu_i-r_i+1) v_\alpha.
 \endaligned$$
 Then by the second formula in (\ref{ind-mod}), considering the coefficient  $(|\mu|-|\br|+\alpha_i)(-\alpha_{i}+\mu_i-r_i+1) \neq 0$, we conclude that  $e_{i0}$ acts injectively on $G_\mu(V(\lambda)^\tau)$ if and only if
 $\mu_i-\l_i\not\in \Z$ and $|\mu|+\l_i\not\in \Z$ for any
$i\in \{1,\dots,n\}$.
\end{proof}

By Theorem \ref{irre-factor2}, the composition length  of  $G_{\mu}(V(\l)^\tau)$ is at most $2$. Now we are ready to give a new classification of irreducible
cuspidal $\sl_{n+1}$-modules.

\begin{theorem}\label{Realization-theorem}[Realization Theorem] If $M$ is an irreducible cuspidal $\sl_{n+1}$-module such that $\mathrm{supp}M=\mu+\Z^n$ for some
$\mu\in \C^n$,
then there is a finite-dimensional irreducible
$\gl_n$-module $V(\l)$ such that $M$ is isomorphic to
an irreducible quotient module of $G_\mu(V(\l)^\tau)$, where $\l\in \Lambda_n^+$ satisfies  the condition (\ref{cus-condition}) and $\tau:W\rightarrow U(\gl_n)$ is  defined by (\ref{x-w-act}).
\end{theorem}

\begin{proof}Since every root vector acts bijectively on
$M$, $M$ can be extended to be a $U_S$-module. Let $N$
be the weight space $M_\mu$. From $[W,\mh_n]=0$,
$WN\subset N$, i.e., $N$ is a $W$-module.
By the isomorphism $U_{S}\cong B \otimes W $,
$M\cong G_\mu(N)$. The irreducibility of
$M$ implies that $N$ is an irreducible $W$-module.
By Theorem \ref{finite-w-mod}, there is a finite-dimensional
irreducible $\gl_n$-module $V(\l)$ such that $N$ is isomorphic to an irreducible $W$-quotient module of $V(\l)^\tau$. Therefore $M$ is isomorphic to
an irreducible $\sl_{n+1}$-quotient module of $G_{\mu}(V(\l)^\tau)$.
By Theorem \ref{irre-factor}, that $N$ is also isomorphic to an irreducible $W$-submodule of $V(\l')^\tau$ for another $\l'\in \Lambda_n^+$. By the similar proof for Lemma \ref{cuspidal-lemm},
$\lambda$ satisfies the condition (\ref{cus-condition}).
\end{proof}

\begin{remark} Recall that  $\mg(0)\cong \gl_n$.  The $\gl_n$-module $V(\lambda)$  can be extended to a module over the parabolic subalgebra $\mathfrak{p}'_n:=\mg(0)\ltimes \mg(-2)$ by defining $\mg(-2)V(\lambda)=0$. The parabolic Verma module $M_{\mathfrak{p}'_n}(\lambda)$ over $\sl_{n+1}$ is
$U(\sl_{n+1})\otimes_{U(\mathfrak{p}'_n)}V(\lambda)$. Denote the unique irreducible   quotient of $M_{\mathfrak{p}'_n}(\lambda)$ by $L_{\mathfrak{p}'_n}(\lambda)$. In \cite{M},
it was shown that  an irreducible cuspidal $\sl_{n+1}$-module is isomorphic to a twisted localization of some $L_{\mathfrak{p}'_n}(\lambda)$, see also Theorem 8.3 in \cite{GG}. Therefore Theorem \ref{Realization-theorem} gives an explicit realization
of some irreducible  twisted localization of some $L_{\mathfrak{p}'_n}(\lambda)$.
\end{remark}

\subsection{Relation between $G_{\mu}(V(\l)^\tau)$ and the Shen-Larsson module}

Finally we discuss  the connection between  $G_{\mu}(V(\l)^\tau)$ and $T(P(\mu),V(\l))$ defined by  Proposition \ref{s-iso}.
For $\mu \in \C^n$, let $P(\mu)=x^\mu\C[x_1^{\pm 1},\dots, x_n^{\pm 1}]$, where
$x^{\mu}=x_1^{\mu_1}\dots x_n^{\mu_n}$.  Then $P(\mu)$ is a $D_n$-module under the natural actions of derivations.
By (\ref{3.1}) , one can check that $T(P(\mu-\l),V(\l))$ is a cuspidal $\sl_{n+1}$-module if and only if $\mu_i-\l_i\not\in \Z$ and $|\mu|+\l_i\not\in \Z$  for any
$i\in \{1,\dots,n\}$.

\begin{proposition}
Let $\l\in \Lambda_n^+$ and  $\mu\in\C^n$ such that $\mu_i-\l_i\not\in \Z$ and $|\mu|+\l_i\not\in \Z$  for any
$i\in \{1,\dots,n\}$. Then
$T(P(\mu-\l),V(\l))\cong G_{\mu}(V(\l)^\tau)$.
\end{proposition}

\begin{proof} For convenience, denote $M=T(P(\mu-\l),V(\l))$.
Since $M$ is a cuspidal $\sl_{n+1}$-module,
we have that $M\cong G_\mu({M_\mu})$. Next we need to show
the $W$-module $M_\mu$ is isomorphic to $V(\l)^\tau$.
Let $V(\lambda)=\oplus_{\alpha\in \C^n}V(\lambda)_\alpha$
be the weight space decomposition of $V(\l)$.
As a vector space, $M_\mu=\oplus_{\alpha\in \mathrm{Supp}V(\l)} x^{\mu-\alpha}\otimes V(\l)_\alpha$.
 Define
an action of $\gl_n$ on $M_\mu$ such that
$$e_{ij}\circ (x^{\mu-\alpha}\otimes v_\alpha)=
\frac{\mu_i-\alpha_i+\delta_{ij}}{\mu_j-\alpha_j+1}
x^{\mu-\alpha+e_j-e_i}\otimes e_{ij}v_\alpha,$$
where $v_\alpha\in V(\l)_\alpha$, $\alpha\in \mathrm{Supp}V(\l)$.
Let $c_\alpha=\prod_{k=1}^n\Gamma(\mu_k-\alpha_k+1)$,
where $\Gamma$ is the Gamma function satisfying
\begin{equation} \label{Gamma}\Gamma(z+1)=z\Gamma(z),\quad z\in \C.
\end{equation}
Define $\Phi: M_\mu\rightarrow V(\l)$ such that
$\Phi(x^{\mu-\alpha}\otimes v_\alpha)=c_\alpha v_\alpha$, which is a vector space isomorphism. By (\ref{Gamma}),
$ \Phi(e_{ij}\circ( x^{\mu-\alpha}\otimes v_\alpha) )
= e_{ij} \Phi( x^{\mu-\alpha}\otimes v_\alpha)
$. Hence $\Phi$ is a $\gl_n$-module isomorphism. By Proposition \ref{s-iso}, the action of $W$ on $M_{\mu}$
satisfies:
$$\aligned
x_{ij}(x^{\mu-\alpha}\otimes v_\alpha)
&=-\alpha_i
x^{\mu-\alpha}\otimes v_\alpha +
\frac{\mu_i-\alpha_i+\delta_{ij}}{\mu_j-\alpha_j+1}x^{\mu-\alpha+e_j-e_i}\otimes e_{ij}v_\alpha\\
&= (e_{ij}-e_{ii})\circ(x^{\mu-\alpha}\otimes v_\alpha)
\endaligned$$
and
 $$\aligned
 \omega_k (x^{\mu-\alpha}\otimes v_\alpha)&
 =\Big(e_{k0}e_{0k}+\sum_{j=1}^n
 (e _{kj}-\frac{\delta_{jk}}{n+1}I_{n+1})(h_j-1)e_{0k}e_{0j}^{-1}\Big)(x^{\mu-\alpha}\otimes v_\alpha)\\
&=\sum_{j=1}^n
\frac{(\mu_k-\alpha_k+\delta_{jk})(\alpha_j-1)}{\mu_j-\alpha_j+1}
x^{\mu-\alpha+e_j-e_k}\otimes e_{kj}v_\alpha\\
&=\sum_{j=1}^n(e_{kj}e_{jj}-e_{kj})\circ(x^{\mu-\alpha}\otimes v_\alpha)
, \endaligned$$
for all $i,j,k\in\{1,\dots,n\}$ with $i\neq j$.
Hence the $W$-module $M_\mu$ is isomorphic to $V(\l)^\tau$.
Therefore $M\cong G_\mu(V(\l)^\tau)$.
\end{proof}

\begin{corollary}\label{Classification-theorem2} If $M$ is an irreducible cuspidal $\sl_{n+1}$-module such that $\mathrm{supp}M=\mu+\Z^n$ for some
$\mu\in \C^n$,
then there is a finite-dimensional irreducible
$\gl_n$-module $V(\l)$ such that $M$ is isomorphic to
an irreducible $\sl_{n+1}$-quotient module of $T(P(\mu-\l),V(\l))$, where $\l\in \C^n$ satisfies that $\mu_i-\l_i\not\in \Z$ and $|\mu|+\l_i\not\in \Z$ for any
$i\in \{1,\dots,n\}$.
\end{corollary}

\vspace{2mm}


%
%

\vspace{4mm}

 \noindent G.L.: School of Mathematics and Statistics,
and  Institute of Contemporary Mathematics,
Henan University, Kaifeng 475004, China. Email: liugenqiang@henu.edu.cn

\vspace{0.2cm}

 \noindent M.L.: School of Mathematics and Statistics,
Henan University, Kaifeng 475004, China. Email: 13849167669@163.com

\end{document}